\documentclass{amsart}
\usepackage{fullpage}
\usepackage{tikz}
\title{Notes on an Elementary Proof for the Stability of Persistence Diagrams}
\author{Primoz Skraba \and Katharine Turner}

\newtheorem{theorem}{Theorem}[section]

\newtheorem{lemma}{Lemma}[section]

\newtheorem{definition}{Definition}[section]
\newtheorem{exercise}{Exercise}[section]
\newtheorem{remark}{Remark}[section]
\newtheorem{assumption}{Assumption}[section]

\newcommand{\dgm}{\mathrm{Dgm}}
\begin{document}
\maketitle
\section{Setup}
These notes provide a self contained algorithmic proof bottleneck stability of persistence diagrams.  The only assumption is familiarity with
the standard persistence algorithm. The proof technique itself is a special case of the proof of  $p$-Wasserstein stability for cellular complexes in  \cite[Section 3]{wasserstein}. The proof is further simplified due to the use of bottleneck stability. Similar ideas can also be found in ~\cite{vineyards}. 

As input, fix a finite simplicial complex $K$, endowed with two functions $f_0,f_1,:K\rightarrow  \mathbb{R}$, which we assume satisfy the following conditions:
\begin{enumerate}
\item The functions are piecewise constant, i.e. the function assign  a function value to each simplex -- so for any simplex $\sigma \in K$, the notion $f_0(\sigma)$ and $f_1(\sigma)$ make sense.
\item The functions are bounded, i.e. for all $\sigma \in K$, $|f_0(\sigma)|<\infty $ and $|f_1(\sigma)|<\infty $  
\item For any $\alpha \in \mathbb{R}$, the sublevel sets $f_0^{-1}(-\infty,\alpha]$  and $f_1^{-1}(-\infty,\alpha]$  are simplicial complexes.
\end{enumerate}
Notice that the above conditions are just needed for the standard persistence algorithm from \cite{algorithm} to be well-defined. The function defines an ordering on the simplices. If each simplex has a unique function value, then the ordering is total (a linear order). Otherwise, we can extend the partial order  to a total order. If two simplices have the same function value, they are ordered according to increasing dimension (to ensure that at each step in the ordering is a valid simplicial complex). If they have the same dimension, an arbitrary ordering can be  chosen (e.g. lexicographical ordering). Our statement will involve the persistence diagrams of the sub-level set filtrations of $f_0$ and $f_1$, which we denote $\dgm(f_0)$ and $\dgm(f_1)$ respectively.

 We recall the following definitions:
 \begin{definition}
 For every point in a diagram $p\in\dgm$,
 \begin{itemize}
 \item $b(p):$ birth time ($x$-coordinate)
 \item $d(p):$ death time ($y$-coordinate)
 \end{itemize} 
 \end{definition}
There are several distances between diagrams which have been defined. but the \emph{bottleneck distance} is perhaps the best known.  The distances are all defined in terms of a matching. A matching in our context is a set map, $\pi$ between the between the points in the two diagrams.  
\begin{definition}
 The bottleneck distance between two diagrams is given by
 $$ d_B(\dgm(f_0),\dgm(f_1)) = \inf\limits_{\pi\in \mathrm{bijections}} \max\limits_{p \in \dgm(f_0)} \max\{b(p) -d(\pi(p)),b(p) -d(\pi(p))\}$$ 
 \end{definition}
 In the above theorem, we take the infimum over all bijections between the diagrams. It is important to observe that we are taking the algorithmic definition of the diagram where the points in the diagram are precisely the pairs of simplices returned by the algorithm. This may include points with multiplicity as well as points on the diagonal, i.e. cycles which immediately become bounded. 
 In the general case, we must allow points to get matched to an arbitrary point on the diagonal of the diagram. However, in this restricted setting, the number of points is always equal and we can consider the distance over all bijections without the addition of the diagonal -- although some of the points may lie on the diagonal (see the following exercises).  
 \begin{exercise}
 Prove that for a fixed simplicial complex, the number of points in the persistence diagram is always the same (regardless of the filtration function).
 \end{exercise}
 \begin{exercise}
 Prove that the bottleneck distance where we allow points to map to the diagonal is smaller than the distance defined above.
 \end{exercise}
 \begin{definition}
 $$||f_0-f_1||_\infty = \max\limits_{\sigma \in K} | f_0(\sigma)-f_1(\sigma) |$$
 \end{definition}
 
 We will  prove the following theorem:
\begin{theorem}\label{thm:main}
$$d_B(\dgm(f_0),\dgm(f_1)) \leq ||f_0-f_1||_\infty$$
where $d_B(\cdot)$  is the bottleneck distance. 
\end{theorem}
\noindent We will use the following three facts about the persistence algorithm:
\begin{enumerate}
\item The persistence algorithm returns the diagram in the form of a list of pairs of simplices $(\sigma, \tau)$, i.e. one pair per point in the diagram.  In fact, we take the pairs as the persistent diagrams. 
\item If the  ordering of simplices does not change, the persistence algorithm returns the same pairs. 
\item Any extension of a partial order to a total order produces the same diagram. 
\end{enumerate}
\section{Proof}
After running the persistence algorithm on the filtration induced by $f$. For each point in the computed persistence diagram $p\in \dgm(f)$, we have the corresponding pivots.
This defines the following map:
$$ \mathrm{piv}(p) = (\sigma,\tau)$$
where $\sigma$ and $\tau$ are the birth and death simplices respectively, i.e. $f(\sigma) = b(p)$ and    $f(\tau) = d(p)$. One way of interpreting this map is that it returns the pivots in the reduced boundary matrix in the algorithm. 
For convenience we use $\mathrm{piv}_b(p)$ and $\mathrm{piv}_d(p)$ to denote the simplices corresponding to birth and death times respectively.

To prove the result, we will linearly interpolate between $f_0$ and $f_1$. For $t\in[0,1]$ for each $\sigma \in K$
$$f_t(\sigma) = (1-t)\cdot f_0(\sigma) + t\cdot f_1(\sigma)$$
\begin{exercise}
Show that for each $t$, the sublevel sets each form a simplicial complex. 
\end{exercise}

For two functions $f_\alpha$ and $f_\beta$ with $\alpha<\beta \in[0,1]$, we say that the ordering of the simplices does not change if for any two simplices  $\sigma, \tau \in K$,
either
$$
f_t(\sigma) \leq f_{t}(\tau) 
$$
for all $t\in  [\alpha,\beta]$, or 
$$
f_t(\sigma) \geq f_{t}(\tau) 
$$
for all $t\in  [\alpha,\beta]$.
\begin{assumption}\label{as:one}
We assume that both $f_0$ and $f_1$  assign a unique value to each simplex, i.e. the ordering in both filtrations is unique. 
\end{assumption}
\begin{lemma}\label{lem:finite}
There are a finite number of values of $0\leq t\leq1$, where there exist simplices $\sigma$ and $\tau$ such that
$$f_t(\sigma) = f_t(\tau)$$
\end{lemma}
\begin{proof}
By assumption, $t=0$ and $t=1$, there are no such simplices. Therefore, the only time two simplices have equivalent values is when their linear interpolations cross as in the following Figure.
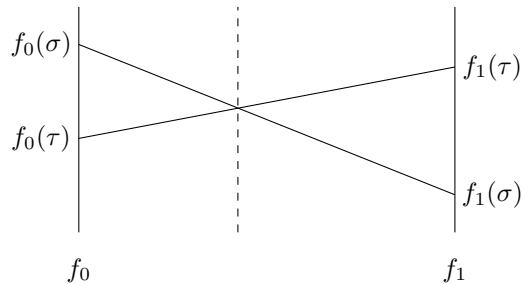
\begin{center}
\begin{tikzpicture}
\draw (0,0) -- (0,3);
\draw (5,0) -- (5,3);
\node  at (0,-0.5) {$f_0$};
\node  at (5,-0.5) {$f_1$};
\draw (0,2.5) -- (5,0.5);
\draw (0,1.25) -- (5,2.2);
\node  at (-0.5,2.5) {$f_0(\sigma)$};
\node  at (-0.5,1.25) {$f_0(\tau)$};
\node  at (5.5,0.5) {$f_1(\sigma)$};
\node  at (5.5,2.2) {$f_1(\tau)$};
\draw[dashed] (2.115,0) -- (2.115,3);
\end{tikzpicture}
\end{center}
 So for every pair of simplices, there is at most one intersection. Assuming there are $n$ simplices, there are at most $\frac{n^2-n}{2}$ intersections. At all other points, the function values will be unique.
\end{proof}

\begin{lemma}\label{lem:easy}
Let $\alpha<\beta \in [0,1]$ such that
\begin{enumerate}
\item There is an $s\in   [\alpha,\beta]$ such that $f_s$ induces a total order.
\item  The ordering of the simplices in $  [\alpha,\beta]$  does not change .
\end{enumerate}
Then
$$ d_B(\dgm(f_\alpha),\dgm(f_{\beta})) \leq ||f_\alpha-f_\beta||_\infty$$ 
\end{lemma}
\begin{proof}
First we must show that there exists a consistent ordering for the persistence algorithm in the interval, i.e. for all points in the diagram, the map $\mathrm{piv}(\cdot)$ is constant.
Choose the (unique) ordering induced at $f_s$. Since the ordering does not change throughout the interval, this is a valid ordering for all $t\in  [\alpha,\beta]$. As this ordering does not change, the persistence pairs returned by the algorithm are the same throughout the interval.

Now we construct a matching  between the diagrams $\mathbf{M}$ which maps points in $\dgm(f_\alpha)$ to points in $\dgm(f_\beta)$. 
$\mathbf{M}$ maps a point $p\in \dgm(f_a)$ to a point $q\in \dgm(f_b)$ if and only if 
$$\mathrm{piv}(p) = \mathrm{piv}(q)$$
Since the bottleneck distance is infimum over all matchings
$$d_B(\dgm(f_\alpha),\dgm(f_{\beta})) \leq \max\limits_{p\in \dgm(f_\alpha)} \max\{ |b(p) - b(\mathbf{M}(p))|, |d(p) - d(\mathbf{M}(p))\} $$
By definition we have  that 
$$b(p) = f_\alpha(\mathrm{piv}_b(p)),  \quad \quad d(p) = f_\alpha(\mathrm{piv}_d(p)),  $$ 
and
$$b(\mathbf{M}(p)) = f_\beta(\mathrm{piv}_b(\mathbf{M}(p))),  \quad \quad d(\mathbf{M}(p)) = f_\beta(\mathrm{piv}_d(\mathbf{M}(p))), $$ 
Substituting in we conclude the proof.
\begin{align*}
d_B(\dgm(f_a),\dgm(f_{b})) &\leq  \max\limits_{p\in \dgm(f_a)} \max\{ |f_a(\mathrm{piv}_b(p)) - f_b(\mathrm{piv}_b(\mathbf{M}(p)))|, f_a(\mathrm{piv}_d(p)) - f_b(\mathrm{piv}_d(\mathbf{M}(p)))| \} \\
&=   \max\limits_{\sigma\in K}  | f_a(\sigma) -f_b(\sigma)|  = ||f_a-f_b||_\infty
\end{align*}
\end{proof}
\begin{exercise}
Prove Lemma~\ref{lem:easy} without the assumption of uniqueness on $f_0$ and $f_1$, i.e. prove Lemma~\ref{lem:easy} without Assumption~\ref{as:one}.
\end{exercise}
\noindent We can now complete the proof.
\begin{proof}[Proof of Theorem~\ref{thm:main}]
Recall by Lemma~\ref{lem:finite}, that there are only finitely many values of $t\in [0,1]$, where the values of any two simplices coincide. We denote these values
$$ t(1) < t(2) <\ldots < t(k)$$
and we set $t(0)=0$ and $t(k+1)=1$. 

We now observe two facts. First,
notice that since we are considering  a linear interpolation so for any $\alpha<\beta \in [0,1]$
$$||f_\alpha - f_\beta||_\infty = (\beta-\alpha) ||f_0-f_1||_\infty $$
Second, if we restrict to  each of the intervals $[t(i),t(i+1)]$ for $i=0,\ldots k$, the conditions of Lemma~\ref{lem:easy} apply. So we have
\begin{align*}
d_B(\dgm(f_{t(i)}),\dgm(f_{t(i+1)})) \leq ||  f_{t(i)} - f_{t(i+1)} ||_\infty
\end{align*}
Using the triangle inequality,
\begin{align*}
d_B(\dgm(f_{0}),\dgm(f_{1})) &\leq \sum_{i=0}^k d_B(\dgm(f_{t(i)}),\dgm(f_{t(i+1)}))  \\
&\leq \sum_{i=0}^k  ||  f_{t(i)} - f_{t(i+1)} ||_\infty \\
& = \sum_{i=0}^k  (t(i+1)- t(i))    ||  f_0 - f_1 ||_\infty  \\
& = (t(k+1) - t(0)) ||  f_0 - f_1 ||_\infty  =  ||  f_0 - f_1 ||_\infty 
\end{align*}
which completes the proof.
\end{proof}
\begin{remark}
One of the important things to realize in this proof is that although we can choose a consistent ordering in each interval, there is necessarily a discontinuity in the choice at the intersection of two intervals (when the ordering is not consistent). The reason we can do this in the proof is that at that each of those points in the interpolation, there is consistent ordering from below and a different consistent ordering from above. At exactly the function value, there are multiple valid orderings, which by Fact (3), yields the same diagram.
\end{remark}
\bibliographystyle{unsrt}
\bibliography{main}

\begin{thebibliography}{1}

\bibitem{wasserstein}
Primoz Skraba and Katharine Turner.
\newblock Wasserstein stability for persistence diagrams.
\newblock {\em arXiv preprint arXiv:2006.16824}, 2020.

\bibitem{vineyards}
David Cohen-Steiner, Herbert Edelsbrunner, and Dmitriy Morozov.
\newblock Vines and vineyards by updating persistence in linear time.
\newblock In {\em Proceedings of the twenty-second annual symposium on
  Computational geometry}, pages 119--126, 2006.

\bibitem{algorithm}
Afra Zomorodian and Gunnar Carlsson.
\newblock Computing persistent homology.
\newblock {\em Discrete \& Computational Geometry}, 33(2):249--274, 2005.

\end{thebibliography}

\end{document}